\documentclass{amsart}
\usepackage{amsfonts,amssymb,amsmath,amsthm,mathrsfs}
\usepackage{url}
\usepackage{enumerate}

\newtheorem{theorem}{Theorem}[section]
\newtheorem{lemma}[theorem]{Lemma}

\theoremstyle{definition}

\numberwithin{equation}{section}
\begin{document}

\title[Marcinkiewicz integral]
{Weighted $L^p$ bounds for the Marcinkiewicz integral}
\author[G. Hu]{Guoen Hu}
\address{Guoen Hu, Department of  Applied Mathematics, Zhengzhou Information Science and Technology Institute\\
Zhengzhou 450001,
P. R. China}
\author[M. Qu]{Meng Qu}
\address{Meng Qu, School of Mathematics and Computer Science, Anhui Normal university, Wuhu 241002, P. R. China}

\email{guoenxx@163.com,\,qumeng@mail.ahnu.edu.cn}
\thanks{The research of the first (corresponding) author was supported
the NNSF of
China under grant $\#$11371370,  the research of the second author was supported
the NNSF of
China under grant $\#$11471033.}

%% NB There should be only one primary classification, and zero or
%more secondary classifications.

\begin{abstract}
Let $\Omega$ be homogeneous of degree zero, have mean value zero and integrable on the unit sphere, and $\mathcal{M}_{\Omega}$ be the higher-dimensional Marcinkiewicz
integral associated with $\Omega$. In this paper, the authors proved that if $\Omega\in L^q(S^{n-1})$ for some $q\in (1,\,\infty]$, then for $p\in (q',\,\infty)$ and $w\in A_{p}(\mathbb{R}^n)$, the bound of $\mathcal{M}_{\Omega}$ on $L^p(\mathbb{R}^n,\,w)$ is less than $C[w]_{A_{p/q'}}^{2\max\{1,\,\frac{1}{p-q'}\}}$.
\end{abstract}
\maketitle

\section{Introduction}

We will work on $\mathbb{R}^n$, $n\geq 2$. Let $M$ be the Hardy-Littlewood maximal operator, and $A_p(\mathbb{R}^n)$ be the weight function class of Muckenhoupt, that is, $$A_{p}(\mathbb{R}^n)=\{w\,\hbox{is\,\,nonnegative\,\,and\,\,locally\,\,integrable\,\,in}\,\,\mathbb{R}^n:\,[w]_{A_p}<\infty\}$$
(see \cite[Chapter 9]{gra2} for the properties of $A_p(\mathbb{R}^n)$), where and in the following, $$[w]_{A_p}:=\sup_{Q}\Big(\frac{1}{|Q|}\int_Qw(x){\rm d}x\Big)\Big(\frac{1}{|Q|}\int_{Q}w^{-\frac{1}{p-1}}(x){\rm d}x\Big)^{p-1},$$
which is called the $A_p$ constant of $w$. In the remarkable work,  Buckley \cite{bu} proved that if $p\in (1,\,\infty)$ and $w\in A_{p}(\mathbb{R}^n)$, then
\begin{eqnarray}\|Mf\|_{L^{p}(\mathbb{R}^n,\,w)}\lesssim_{n,\,p}[w]_{A_p}^{\frac{1}{p-1}}\|f\|_{L^{p}(\mathbb{R}^n,\,w)}.\end{eqnarray}
Moreover, the estimate (1.1) is sharp since the exponent $1/(p-1)$ can not be replaced by a smaller one. Since then,
the sharp dependence of the weighted estimates of  singular integral operators in terms of the $A_p(\mathbb{R}^n)$ constant has been considered by many authors.  Petermichl \cite{pet1,pet2} solved this question for Hilbert transform and Riesz transform.   Hyt\"onen \cite{hy}
proved that  for a  Calder\'on-Zygmund operator $T$ and $w\in A_2(\mathbb{R}^n)$,
\begin{eqnarray}\|Tf\|_{L^{2}(\mathbb{R}^n,\,w)}\lesssim_{n}[w]_{A_2}\|f\|_{L^{2}(\mathbb{R}^n,\,w)}.\end{eqnarray}
This solved  the so-called $A_2$ conjecture. From the estimate (1.2) and the extrapolation theorem in \cite{dra}, we know that
for a Calder\'on-Zygmund operator $T$, $p\in (1,\,\infty)$ and $w\in A_p(\mathbb{R}^n)$,
\begin{eqnarray}\|Tf\|_{L^{p}(\mathbb{R}^n,\,w)}\lesssim_{n,\,p}[w]_{A_p}^{\max\{1,\,\frac{1}{p-1}\}}\|f\|_{L^{p}(\mathbb{R}^n,\,w)}.\end{eqnarray}
In \cite{ler1}, Lerner  gave a very simple proof of (1.3) by  controlling the Calder\'on-Zygmund operator using sparse operators.

Fairly recently, Hyt\"onen, Roncal and Tapiola \cite{hyta} considered the weighted  bounds of rough homogeneous singular integral operator defined by
$$T_{\Omega}f(x)={\rm p. \,v.}\int_{\mathbb{R}^n}\frac{\Omega(y')}{|y|^n}f(x-y){\rm d}y=\lim_{\epsilon\rightarrow 0\atop{R\rightarrow\infty}}\int_{\epsilon<|x-y|<R}\frac{\Omega(y')}{|y|^n}f(x-y){\rm d}y,$$
where $\Omega$ is homogeneous of degree zero, integrable on the unit sphere $S^{n-1}$ and has mean value zero.
By a quantitative weighted estimate for the $\omega$-Calder\'on-Zygmund operator,  approximation to the identity and interpolation, Hyt\"onen, Roncal  and Tapiola
(see Theorem 1.4 in \cite{hyta}) proved that
\begin{theorem}\label{1.1}
Let $p\in (1,\,\infty)$ and $w\in A_{p}(\mathbb{R}^n)$. Then
\begin{eqnarray}\|T_{\Omega}f\|_{L^p(\mathbb{R}^n,\,w)}\lesssim \|\Omega\|_{L^{\infty}(S^{n-1})}\{w\}_{A_{p}}(w)_{A_{p}}\|f\|_{L^p(\mathbb{R}^n,\,w)},\end{eqnarray}
where and in the following, for $p,\,r\in (1,\,\infty)$,
$$\{w\}_{A_p,\,r}=[w]_{A_p}^{\frac{1}{r}}\max\{[w]_{A_{\infty}}^{\frac{1}{r'}},\,[w^{1-p'}]_{A_{\infty}}^{\frac{1}{r}}\},$$
and
$$(w)_{A_p}=\max\{[w]_{A_{\infty}},\,[w^{1-p'}]_{A_{\infty}}\}.$$
\end{theorem}

Now we consider the Marcinkiewicz integral operator. Let $\Omega$ be
homogeneous of degree zero, integrable and have mean value zero on
the unit sphere $S^{n-1}$. Define the Marcinkiewicz integral
operator $\mathcal{M}_\Omega$ by
\begin{eqnarray}\mathcal{M}_\Omega(f)(x)= \Big(\int_0^\infty|F_{\Omega,\,
t}f(x)|^2\frac{{\rm d}t}{t^3}\Big)^{1/2},\end{eqnarray}where $$F_{\Omega,
t}f(x)= \int_{|x-y|\leq t}\frac{\Omega(x-y)}{|x-y|^{n-1}}f(y){\rm d}y $$ for
$f\in \mathcal{S}(\mathbb{R}^n)$.  Stein \cite{st}
proved that if $\Omega\in {\rm Lip}_{\gamma}(S^{n-1})$ with
$\gamma\in (0,\,1]$, then $\mathcal{M}_\Omega$ is bounded on
$L^p(\mathbb{R}^n)$ for $p\in (1,\,2]$. Benedek, Calder\'on and
Panzon showed that the $L^p(\mathbb{R}^n)$ boundedness $(p\in
(1,\,\infty)$) of $\mathcal{M}_\Omega$ holds true  under the
condition that $\Omega\in C^1(S^{n-1})$.  Walsh \cite{wal} proved
that for each $p\in (1,\,\infty)$, $\Omega\in L(\ln L)^{1/r}(\ln \ln
L)^{2(1-2/r')}(S^{n-1})$ is a sufficient condition such that
$\mathcal{M}_\Omega$ is bounded on $L^{p}(\mathbb{R}^n)$, where
$r=\min\{p,\,p'\}$ and $p'=p/(p-1)$. Ding, Fan and Pan \cite{dfp}
proved that if $\Omega\in H^1(S^{n-1})$ (the Hardy space on
$S^{n-1}$), then $\mathcal{M}_\Omega$ is bounded on
$L^p(\mathbb{R}^n)$  for all $p\in (1,\,\infty)$; Al-Salmam et al.
\cite{aacp} proved that $\Omega\in L(\ln L)^{1/2}(S^{n-1})$ is a
sufficient condition such that $\mathcal{M}_\Omega$ is bounded on
$L^p(\mathbb{R}^n)$ for all $p\in(1,\,\infty)$. Ding, Fan and Pan \cite{dfp2} considered the boundedness on weighted $L^p({\mathbb R}^n)$
with $A_p(\mathbb{R}^n)$ when $\Omega\in L^q(S^{n-1})$ for some $q\in (1,\,\infty]$. For other works about the operator $\mathcal{M}_{\Omega}$, see \cite{a, cfp, dfp,dxy} and the related references therein.

The purpose of this paper  is to establish an analogy of (1.4) for the Marcinkiewicz integral operator with kernel $\Omega\in L^q(S^{n-1})$ for some $q\in (1,\,\infty]$. We remark that in this paper, we are very much motivated by \cite{hyta}, some ideas are from
Lerner's recently paper \cite{ler3}.
Our main result can be stated as follows.

\begin{theorem}\label{t1.1}
Let $\Omega$ be homogeneous of degree zero, have mean value zero on $S^{n-1}$, and $\Omega\in L^q(S^{n-1})$ for some $q\in (1,\,\infty]$. Let $p\in (q',\,\infty)$ and $w\in A_{p/q'}(\mathbb{R}^n)$. Then
$$\|\mathcal{M}_{\Omega}f\|_{L^p(\mathbb{R}^n,\,w)}\lesssim \|\Omega\|_{L^q(S^{n-1})}\{w\}_{A_{p/q'},\,p}(w)_{A_{p/q'}}\|f\|_{L^p(\mathbb{R}^n,\,w)}.$$
In particular,
$$\|\mathcal{M}_{\Omega}f\|_{L^p(\mathbb{R}^n,\,w)}\lesssim \|\Omega\|_{L^q(S^{n-1})}[w]_{A_{p/q'}}^{2\max\{1,\,\frac{1}{p-q'}\}}\|f\|_{L^p(\mathbb{R}^n,\,w)}.$$
\end{theorem}
We make some conventions. In what follows, $C$ always denotes a
positive constant that is independent of the main parameters
involved but whose value may differ from line to line. We use the
symbol $A\lesssim B$ to denote that there exists a positive constant
$C$ such that $A\le CB$.   For a set $E\subset\mathbb{R}^n$,
$\chi_E$ denotes its characteristic function.

\section{Proof of Theorem 1.2}
Recall that  the standard dyadic grid in $\mathbb{R}^n$ consists of all cubes of the form $$2^{-k}([0,\,1)^n+j),\,k\in \mathbb{Z},\,\,j\in\mathbb{Z}^n.$$
Denote the standard grid by $\mathcal{D}$.

As usual, by a general dyadic grid $\mathscr{D}$,  we mean a collection of cube with the following properties: (i) for any cube $Q\in \mathscr{D}$, it side length $\ell(Q)$ is of the form $2^k$ for some $k\in \mathbb{Z}$; (ii) for any cubes $Q_1,\,Q_2\in \mathscr{D}$, $Q_1\cap Q_2\in\{Q_1,\,Q_2,\,\emptyset\}$; (iii) for each $k\in \mathbb{Z}$, the cubes of side length $2^k$ form a partition of $\mathbb{R}^n$.

Let $\eta\in (0,\,1)$ and $\mathcal{S}$ be a family of cubes. We say that $\mathcal{S}$ is $\eta$-sparse,  if for each fixed $Q\in \mathcal{S}$, there exists a measurable subset $E_Q\subset Q$, such that $|E|\geq \eta|Q|$ and $\{E_{Q}\}$ are pairwise disjoint.
Associated with  the sparse family $\mathcal{S}$, we define the sparse operator $\mathcal{A}_{\mathcal{S}}$ by
$$\mathcal{A}_{\mathcal{S}}f(x)=\sum_{Q\in\mathcal{S}}\langle f\rangle_{Q}\chi_{Q}(x),$$
where and in the following, $\langle f\rangle_Q=\frac{1}{|Q|}\int_{Q}f(y)dy.$ For $r\in (0,\,\infty)$, let
$\mathcal{A}_{\mathcal{S}}^{r}$ be the operator defined by
$$\mathcal{A}_{\mathcal{S}}^{r}f(x)=\Big\{\sum_{Q\in\mathcal{S}}\big(\langle f\rangle_{Q}\big)^r\chi_{Q}(x)\Big\}^{1/r},$$

The following result was proved by Hyt\"onen and  Lacey \cite{hyla}, see also Hyt\"onen and Li \cite{hyli}.
\begin{lemma}\label{l2.1} Let $p\in (1,\,\infty)$ and $r\in (0,\,p)$, $w\in A_p(\mathbb{R}^n)$.
Then for a sparse family $\mathcal{S}\subset\mathscr{D}$ with $\mathscr{D}$ a dyadic grid,
$$\|\mathcal{A}_{\mathcal{S}}^{r}f\|_{L^p(\mathbb{R}^n,\,w)}\lesssim [w]_{A_p}^{1/p}
\big([w]_{A_{\infty}}^{\frac{1}{r}-\frac{1}{p}}+[w^{-\frac{1}{p-1}}]_{A_{\infty}}^{\frac{1}{p}}\big)
\|f\|_{L^{p}(\mathbb{R}^n,\,w)}.$$
\end{lemma}

Let $\Omega$ be homogeneous of degree zero, integrable on $S^{n-1}$. For $t\in [1,\,2]$ and $j\in\mathbb{Z}$, set
\begin{eqnarray}K^j_t(x)=\frac{1}{2^j}\frac{\Omega(x)}{|x|^{n-1}}\chi_{\{2^{j-1}t<|x|\leq 2^jt\}}(x).\end{eqnarray}
As it was proved in \cite{drf}, if $\Omega\in L^q(S^{n-1})$ for some
$q\in (1,\,\infty]$, then there exists a constant $\alpha\in
(0,\,1)$ such that for $t\in[1,\,2]$ and
$\xi\in\mathbb{R}^n\backslash\{0\}$,
\begin{eqnarray}|\widehat{K^j_t}(\xi)|\lesssim \|\Omega\|_{L^{q}(S^{n-1})}\min\{1,\,|2^j\xi|^{-\alpha}\}.\end{eqnarray} Here and in the following for $h\in\mathcal{S}'(\mathbb{R}^n)$, $\widehat{h}$ denotes the Fourier transform of $h$. Moreover, if $\int_{S^{n-1}}\Omega(x'){\rm d}x'=0$, then
\begin{eqnarray}|\widehat{K^j_t}(\xi)|\lesssim \|\Omega\|_{L^{1}(S^{n-1})}\min\{1,\,|2^j\xi|\}.\end{eqnarray}

In the following, we assume that $\|\Omega\|_{L^{q}(S^{n-1})}=1$. Let
$$\widetilde{\mathcal{M}}_{\Omega}f(x)=\Big(\int^2_1\sum_{j\in\mathbb{Z}}\big|F_{j}f(x,\,t)\big|^2{\rm d}t\Big)^{1/2},$$with
$$F_jf(x,\,t)=\int_{\mathbb{R}^n}K^j_t(x-y)f(y){\rm d}y.$$
A trivial computation leads to that
\begin{eqnarray}\mathcal{M}_{\Omega}f(x)\approx \widetilde{\mathcal{M}}_{\Omega}f(x).
\end{eqnarray}

 Let $\phi\in
C^{\infty}_0(\mathbb{R}^n)$ be a nonnegative function such that
$\int_{\mathbb{R}^n}\phi(x){\rm d}x=1$, ${\rm
supp}\,\phi\subset\{x:\,|x|\leq 1/4\}$. For $l\in \mathbb{Z}$, let
$\phi_l(y)=2^{-nl}\phi(2^{-l}y)$. It is easy to verify that for any
$\varsigma\in (0,\,1)$,
\begin{eqnarray}|\widehat{\phi_l}(\xi)-1|\lesssim \min\{1,\,|2^l\xi|^{\varsigma}\}.\end{eqnarray}Let
$$F_{j}^lf(x,\,t)=\int_{\mathbb{R}^n}K^j_t*\phi_{j-l}(x-y)f(y)\,{\rm d}y.
$$
Define the operator $\widetilde{\mathcal{M}}_{\Omega}^l$ by
\begin{eqnarray}\widetilde{\mathcal{M}}_{\Omega}^lf(x)=\Big(\int^2_1\sum_{j\in\mathbb{Z}}\big|F_{j}^lf(x,\,t)\big|^2
{\rm d}t\Big)^{1/2}.\end{eqnarray}
By Fourier transform estimates (2.2) and (2.5), and Plancherel's theorem, we have
that for some positive constant $\theta$ depending only on $n$,
\begin{eqnarray}
\|\widetilde{\mathcal{M}}_{\Omega}f-\widetilde{\mathcal{M}}_{\Omega}^{l}f\|_{L^2(\mathbb{R}^n)}^2
&=&\int^2_1\Big\|\Big(\sum_{j\in\mathbb{Z}}\big|F_lf(\cdot,\,t)-F_{j}^lf(\cdot,\,t)\big|^2\Big)^{\frac{1}{2}}\Big\|_{L^2(\mathbb{R}^n)}^2{\rm
d}t\\
&=&\int^2_1\sum_{j\in\mathbb{Z}}\int_{\mathbb{R}^n}|\widehat{K_t^j}(\xi)|^2|1-\widehat{\phi_{j-l}}(\xi)|^2|\widehat{f}(\xi)|^2{\rm
d}\xi{\rm d}t\nonumber\\
&\lesssim&2^{-2\theta l}\|f\|_{L^2(\mathbb{R}^n)}^2.\nonumber
\end{eqnarray}

\begin{lemma}\label{l3.1}
Let $\Omega$ be
homogeneous of degree zero and belong to $L^q(S^{n-1})$ for some $q\in (1,\,\infty]$, $K_t^j$ be defined as in (2.1). Then for
$l\in\mathbb{N}$, $R>0$ and  $y\in \mathbb{R}^n$ with $|y|<R/4$,
\begin{eqnarray*}&&\sum_{j\in\mathbb{Z}}\Big(\int_{2^kR<|x|\leq 2^{k+1}R}\sup_{t\in [1,\,2]}\Big|K^j_{t}*\phi_{j-l}(x+y)-K^j_{t}*\phi_{j-l}(x)\Big|^{q}{\rm d}x\Big)^{\frac{1}{q}}\\
&&\quad\lesssim \frac{1}{(2^kR)^{n/q'}}\min\{1,\,2^l\frac{|y|}{2^kR}\}.\end{eqnarray*}
\end{lemma}

\begin{proof} We will employ the idea from \cite{wat}. It is obvious that
$$\|\phi_{j-l}(\cdot+y)-\phi_{j-l}(\cdot)\|_{L^{q'}(\mathbb{R}^n)}\lesssim 2^{(l-j)n/q}\min\{1,\,2^{l-j}|y|\}.$$Observe that
$$\sup_{t\in [1,\,2]}\Big|K^j_{t}*\phi_{j-l}(x+y)-K^j_{t}*\phi_{j-l}(x)\Big|\lesssim \int_{\mathbb{R}^n}\widetilde{K}^j(z)\big|\phi_{j-l}
(x+y-z)-\phi_{j-l}(x-z)\big|{\rm d}z,$$
with $\widetilde{K}^j(z)=|z|^{-n}|\Omega(z)|\chi_{\{2^{j-2}\leq |z|\leq 2^{j+2}\}}(z).$
Thus, by the fact ${\rm supp}\,K_t^j*\phi_{j-l}\subset \{x\in\mathbb{R}^n:\,2^{j-2}\leq |x|\leq 2^{j+2}\}$,  we deduce that
\begin{eqnarray*}
&&\sum_{j\in\mathbb{Z}}\Big(\int_{2^kR<|x|\leq 2^{k+1}R}\sup_{t\in[1,\,2]}\Big|K^j_{t}*\phi_{j-l}(x+y)-K^j_{t}*\phi_{j-l}(x)\Big|^{q}{\rm d}x\Big)^{\frac{1}{q}}\\
&&\quad\lesssim\sum_{j\in\mathbb{Z}:2^j\approx
2^kR}\|\widetilde{K}^j\|_{L^q(\mathbb{R}^n)}\|\phi_{j-l}(\cdot+y)-\phi_{j-l}(\cdot)\|_{L^{q'}(\mathbb{R}^n)}\\
&&\quad\lesssim 2^{jn/q'}\min\{1,\,2^{l}\frac{|y|}{2^kR}\}.
\end{eqnarray*}
This completes the proof of Lemma \ref{l3.1}.
\end{proof}

\begin{lemma}\label{l3.2}Let $\Omega$ be homogeneous of degree zero and have mean value zero.
Suppose that $\Omega\in L^q(S^{n-1})$ for some $q\in (1,\,\infty]$. Then for $l\in\mathbb{N}$, $\widetilde{\mathcal{M}}_{\Omega}^l$
is bounded  from $L^1(\mathbb{R}^n)$ to $L^{1,\,\infty}(\mathbb{R}^n)$ with bound $C_{n,\,p}l$.
\end{lemma}
\begin{proof} The proof is fairly standard. For the sake of self-contained, we present the proof here.  Our goal is to show that for any $\lambda>0$,
\begin{eqnarray}\big|\big\{x\in\mathbb{R}^n:\,\widetilde{\mathcal{M}}_{\Omega}^lf(x)>\lambda\big\}\big|\lesssim l\lambda^{-1}\|f\|_{L^1(\mathbb{R}^n)}.\end{eqnarray}
For each fixed $\lambda>0$,  applying the Calder\'on-Zygmund decomposition to $|f|$ at level $\lambda$,
we obtain a sequence of cubes $\{Q_i\}$ with disjoint interiors, such that
$$\lambda<\frac{1}{|Q_i|}\int_{Q_i}|f(y)|{\rm d}y\le 2^n\lambda,$$
and  $|f(y)|\lesssim \lambda$ for a. e. $y\in\mathbb{R}^n\backslash\big(\cup_{i}Q_i\big).$ Set
$$g(y)=f(y)\chi_{\mathbb{R}^n\backslash\cup_iQ_i}(y)+\sum_{i}\langle f\rangle_{Q_i}\chi_{Q_i}(y),$$
$$b(y)=\sum_{i}b_i(y),\,\,\hbox{with}\,\,
b_i(y)=\big(b_{i}(y)-\langle f\rangle_{Q_i}\big)\chi_{Q_i}(y).$$
By (2.7) and the $L^2(\mathbb{R}^n)$ boundedness of $\widetilde{\mathcal{M}}_{\Omega}$, we know that $\widetilde{\mathcal{M}}_{\Omega}^l$ is also bounded on $L^2(\mathbb{R}^n)$ with bound independent of $l$. Therefore,
$$|\{x\in\mathbb{R}^n:\,\widetilde{\mathcal{M}}_{\Omega}^lg(x)>\lambda/2\}|\lesssim\lambda^{-2}
\|\widetilde{\mathcal{M}}_{\Omega}^lg\|_{L^2(\mathbb{R}^n)}^2\lesssim\lambda^{-1}\|f\|_{L^1(\mathbb{R}^n)}.$$
Let $E_{\lambda}=\cup_i4nQ_i$. It is obvious that $|E_{\lambda}|\lesssim \lambda^{-1}\|f\|_{L^1(\mathbb{R}^n)}$. The proof of (2.8) is now reduced to proving that
\begin{eqnarray}|\{x\in\mathbb{R}^n\backslash E_{\lambda}:\,\widetilde{\mathcal{M}}_{\Omega}^lb(x)>\lambda/2\}|\lesssim l\lambda^{-1}\|f\|_{L^1(\mathbb{R}^n)}.
\end{eqnarray}

We now prove (2.9). For each fixed cube $Q_i$, let $y_i$ be the center of $Q_i$. For $x,\,y,\,z\in\mathbb{R}^n$,  set
$$S_{t}^{j,\,l}(x;\,y,\,z)=|K_t^j*\phi_{j-l}(x-y)-K_t^j*\phi_{j-l}(x-z)|.
$$
A trivial computation involving  Minkowski's inequality and vanishing moment of $b_i$ gives us that for $x\in \mathbb{R}^n$,
\begin{eqnarray*}
\widetilde{\mathcal{M}}_{\Omega}^lb(x)&\le &\sum_{i}\Big(\int^2_1\sum_j\Big(
\int_{\mathbb{R}^n}S_{t}^{j,\,l}(x;\,y,\,y_i)|b_i(y)|{\rm d}y\Big)^2{\rm d}t\Big)^{\frac{1}{2}}\\
&\le &\sum_i\sum_j\int_{\mathbb{R}^n}\Big(\int^2_1\{S_{t}^{j,\,l}(x;\,y,\,y_i)\}^2{\rm d}t\Big)^{\frac{1}{2}}|b_i(y)|{\rm d}y\\
&\le&\sum_i\sum_j\int_{\mathbb{R}^n}\sup_{t\in [1,\,2]}S_{t}^{j,\,l}(x;\,y,\,y_i)|b_i(y)|{\rm d}y.
\end{eqnarray*}
On the other hand, we get from Lemma \ref{l3.1} that
\begin{eqnarray*}
&&\sum_{j\in\mathbb{Z}}\int_{\mathbb{R}^n\backslash E_{\lambda}}\sup_{t\in[1,2]}S_{t}^{j,l}(x;y,y_i){\rm d}x\\
&&\quad=\sum_{k=1}^{l}\sum_{j\in\mathbb{Z}}\int_{2^{k+2} nQ_i\backslash 2^{k+1}nQ_i}\sup_{t\in[1,2]}S_{t}^{j,l}(x;y,y_i){\rm d}x\\
&&\qquad+\sum_{k=l+1}^{\infty}\sum_{j\in\mathbb{Z}}\int_{2^{k+2} nQ_i\backslash 2^{k+1}nQ_i}\sup_{t\in[1,2]}S_{t}^{j,\,l}(x;\,y,\,y_i){\rm d}x\lesssim l.
\end{eqnarray*}
This, in turn leads to that
\begin{eqnarray*}
\int_{\mathbb{R}^n\backslash E_{\lambda}}\widetilde{\mathcal{M}}_{\Omega}^lb(x){\rm d}x&\le& \sum_{i}\sum_j\int_{\mathbb{R}^n}\int_{\mathbb{R}^n\backslash E_{\lambda}}\sup_{t\in[1,2]}S_{t}^{j,\,l}(x;\,y,\,y_i){\rm d}x|b_i(y)|{\rm d}y\\
&\lesssim&l\int_{\mathbb{R}^n}|f(y)|{\rm d}y.
\end{eqnarray*}
The inequality (2.9) now follows directly. This completes the proof of Lemma \ref{l3.2}.
\end{proof}
\begin{lemma}\label{l3.3}Let $T$ be a sublinear operator. Set
$$T^*f(x)=\sup_{Q\ni x}\sup_{\xi\in Q}|T(f\chi_{\mathbb{R}^n\backslash 3Q})(\xi)|.$$ Suppose that for $1\leq s\leq r<\infty$,
$T$ is bounded from $L^{s}(\mathbb{R}^n)$ to $L^{s,\,\infty}(\mathbb{R}^n)$ and $T^*$ is bounded from $L^{r}(\mathbb{R}^n)$ to $L^{r,\,\infty}(\mathbb{R}^n)$.  Then, for every compactly supported function $f\in L^r(\mathbb{R}^n)$, there exists  $\sigma$-sparse families $\mathcal{S}_j\subset\mathscr{D}_j$ with $\mathscr{D}_j$ general dyadic grid and $j=1,\,\dots,\,3^n$, such that for a.\,e. $x\in \mathbb{R}^n$,
\begin{eqnarray}\quad|Tf(x)| \leq C_{n,s,r}\big(\|T\|_{L^s\rightarrow L^{s,\infty}}+\|T^*\|_{L^r\rightarrow L^{r,\infty}}\big)\sum_{j=1}^{3^n}\mathcal{A}_{\mathcal{S}_j,\,r}(|f|)(x),
\end{eqnarray}
where $\sigma=\sigma_n\in (0,\,1)$ and $\mathcal{A}_{\mathcal{S},\,r}f(x)=\sum_{Q\in\mathcal{S}}\langle |f|^r\rangle_Q^{\frac{1}{r}}\chi_{Q}(x).$
\end{lemma}
For the proof of Lemma \ref{l3.3}, see  Theorem 4.2 and Remark 4.3 in \cite{ler3}.
\begin{lemma}\label{l3.4} Let $\Omega$ be homogeneous of degree zero and have mean value zero.
Suppose that $\Omega\in L^q(S^{n-1})$ for some $q\in (1,\,\infty]$. If $p\in (q',\,\infty]$ and $w\in A_{p/q'}(\mathbb{R}^n)$, then
for $l\in\mathbb{N}$,
$$\|\widetilde{\mathcal{M}}_{\Omega}^lf\|_{L^p(\mathbb{R}^n,\,w)}\lesssim l\{w\}_{A_{p/q'},\,p}\|f\|_{L^p(\mathbb{R}^n,\,w)}.$$
\end{lemma}
\begin{proof}  For $j_0\in\mathbb{Z}$ and $l\in\mathbb{N}$, let $\mathcal{N}^{l,\,j_0}_{\Omega}$ be the operator defined by
$$\mathcal{N}^{l,\,j_0}_{\Omega}f(x)=\Big(\int_1^2\sum_{j\leq j_0}|F_j^lf(x,\,t)|^2{\rm d}t\Big)^{1/2}.$$
We claim that for $p\in (1,\,\infty)$,
\begin{eqnarray}
\|\mathcal{N}^{l,\,j_0}_{\Omega}f\|_{L^p(\mathbb{R}^n)}\lesssim l\|f\|_{L^p(\mathbb{R}^n)}.
\end{eqnarray}
To this aim, we first note that if ${\rm supp}\,f\subset Q$ for a cube $Q$ having side length $2^{j_0}$, then
${\rm supp}\,\mathcal{N}^{l,\,j_0}_{\Omega}f\subset 20\sqrt n Q$. On the other hand, if $\{Q_k\}_k$ is a sequence of cubes with disjoint interiors and having side length $2^{j_0}$, then the cubes $\{20\sqrt nQ_k\}_k$ have bounded overlaps. Thus, we may assume that ${\rm supp}\,f\subset Q$ with $Q$ a cube  having side length $2^{j_0}$. For such a $f\in L^p(\mathbb{R}^n)$, we see that if $x\in 20\sqrt nQ$, then
$$\Big(\int_{1}^2\sum_{j>j_0+20n}|F_j^lf(x,\,t)|^2{\rm d}t\Big)^{1/2}=0.$$
Therefore, for $x\in 20\sqrt nQ$,
\begin{eqnarray*}\mathcal{N}_{\Omega}^{l,j_0}f(x)\leq\widetilde{\mathcal{M}}_{\Omega}^lf(x)+
\Big(\int_1^2\sum_{j_0<j<j_0+20n}|F_j^lf(x,\,t)|^2{\rm d}t\Big)^{\frac{1}{2}}\lesssim\widetilde{\mathcal{M}}_{\Omega}^lf(x)+M_{\Omega}Mf(x),
\end{eqnarray*}
where and in the following, $M_{\Omega}$ is the maximal operator defined by
$$M_{\Omega}h(x)=\sup_{r>0}\frac{1}{|B(x,\,r)|}\int_{B(x,\,r)}|\Omega(x-y)||h(y)|{\rm d}y.$$By the method of rotation of Calder\'on and Zygmund, we know that $M_{\Omega}$ is bounded on $L^p(\mathbb{R}^n)$. The desired $L^p(\mathbb{R}^n)$ estimate for $\mathcal{N}_{\Omega}^{l,\,j_0}$ then follows directly.

We now prove that the operator
$$S^lf(x)=\sup_{Q\ni x}\sup_{\xi\in Q}|\widetilde{\mathcal{M}}_{\Omega}^l(f\chi_{\mathbb{R}^n\backslash 3Q})(\xi)|$$
is bounded from $L^{q'}(\mathbb{R}^n)$ to $L^{q',\infty}(\mathbb{R}^n)$ with bound $Cl$. In fact,
let $Q\subset \mathbb{R}^n$ be a cube and
$x,\,\xi\in Q$, and denote by $B_x$ the closed ball centered
at $x$ with  radius $2 {\rm diam}\,Q$. Then $3Q\subset B_x$, and we can write
\begin{eqnarray*}|\widetilde{\mathcal{M}}_{\Omega}^l(f\chi_{R^n\backslash 3Q})(\xi)| &\leq & |\widetilde{\mathcal{M}}_{\Omega}^l(f\chi_{\mathbb{R}^n\backslash B_x})(\xi)- \widetilde{\mathcal{M}}_{\Omega}^l(f\chi_{\mathbb{R}^n\backslash B_x})(x)
|\\
&&+ |\widetilde{\mathcal{M}}_{\Omega}^l(f\chi_{B_x\backslash 3Q})(\xi)| + |\widetilde{\mathcal{M}}_{\Omega}^l(f\chi_{\mathbb{R}^n\backslash B_x})(x)|.
\end{eqnarray*}
It is obvious that
\begin{eqnarray*}
&&|\widetilde{\mathcal{M}}_{\Omega}^l(f\chi_{\mathbb{R}^n\backslash B_x})(\xi)- \widetilde{\mathcal{M}}_{\Omega}^l(f\chi_{\mathbb{R}^n\backslash B_x})(x)
|\\
&&\quad\leq \Big(\int^2_1\sum_{j\in\mathbb{Z}}\Big|\int_{\mathbb{R}^n}R_{t}^{j,\,l}(x;\,y,\,\xi)f(y)\chi_{\mathbb{R}^n\backslash B_x}(y){\rm d}y\Big|^2{\rm d}t\Big)^{\frac{1}{2}},
\end{eqnarray*}
with
$$R_{t}^{j,\,l}(x;\,y,\,\xi)=|K_t^j*\phi_{j-l}(x-y)-K_t^j*\phi_{j-l}(\xi-y)|.$$
A trivial computation involving H\"older's inequality gives us that
\begin{eqnarray}
&&\sum_{j\in\mathbb{Z}}\Big|\int_{\mathbb{R}^n}R_{t}^{j,\,l}(x;\,y,\,\xi)f(y)\chi_{\mathbb{R}^n\backslash B_x}(y){\rm d}y\Big|\\
&&\quad\leq \sum_j\sum_{k=1}^{\infty}\Big(\int_{2^kB_x\backslash 2^{k-1}B_x}\sup_{t\in[1,2]}|R_{t}^{j,\,l}(x;y,\xi)|^q{\rm d}y\Big)^{\frac{1}{q}}\Big(\int_{2^kB_x}|f(y)|^{q'}{\rm d}y\Big)^{\frac{1}{q'}}\nonumber\\
&&\quad\lesssim lM_{q'}f(x),\nonumber
\end{eqnarray}
where $M_{q'}f(x)=\{M(|f|^{q'})(x)\big\}^{1/q'}$. For each fixed  $t\in [1,\,2]$ and $j\in\mathbb{Z}$ with $2^j\approx {\rm diam}\,Q$,
\begin{eqnarray}|F_j^l(f\chi_{B_x\backslash 3Q})(x,\,t)|\le \|K^j_t*\phi_{l-j}\|_{L^q(\mathbb{R}^n)}\|f\chi_{B_x}\|_{L^{q'}(\mathbb{R}^n)}\lesssim M_{q'}f(x).
\end{eqnarray}
Recall that ${\rm supp}\,K_t^j*\phi_{j-l}\subset \{x\in\mathbb{R}^n:\,2^{j-2}\leq |x|\leq 2^{j+2}\}$. It then follows that
\begin{eqnarray}
|\widetilde{\mathcal{M}}_{\Omega}^l(f\chi_{B_x\backslash 3Q})(\xi)|&=&\Big(\int^2_1\sum_j|F_j^l(f\chi_{B_x\backslash 3Q})(x,\,t)|^2{\rm d}t\Big)^{\frac{1}{2}}\\
&\le &\sum_{j:2^j\approx {\rm diam}\,Q}\Big(\int^2_1|F_j^l(f\chi_{B_x\backslash 3Q})(x,\,t)|^2{\rm d}t\Big)^{\frac{1}{2}}\nonumber\\
&\lesssim  & M_{q'}f(x).\nonumber
\end{eqnarray}
To estimate $\widetilde{\mathcal{M}}_{\Omega}^l(f\chi_{\mathbb{R}^n\backslash B_x})$, write
\begin{eqnarray*}
\widetilde{\mathcal{M}}_{\Omega}^l(f\chi_{\mathbb{R}^n\backslash B_x})(x)&\leq & \widetilde{\mathcal{M}}_{\Omega}^lf(x)+\Big(\int^2_1\sum_{j\in \mathbb{Z}}|F_j^l(f\chi_{B_x})(x,\,t)|^2{\rm d}t\Big)^{\frac{1}{2}}\\
&=&\widetilde{\mathcal{M}}_{\Omega}^lf(x)+\Big(\int^2_1\sum_{j: 2^j\leq 4{\rm diam}\,Q}|F_j^l(f\chi_{B_x})(x,\,t)|^2{\rm d}t\Big)^{\frac{1}{2}}\\
&\leq &\widetilde{\mathcal{M}}_{\Omega}^lf(x)+\Big(\int^2_1\sum_{j: 2^j\leq 4{\rm diam}\,Q}|F_j^lf(x,\,t)|^2{\rm d}t\Big)^{\frac{1}{2}}\\
&&+\Big(\int^2_1\sum_{j: 2^j\leq 4{\rm diam}\,Q}|F_j^l(f\chi_{\mathbb{R}^n\backslash B_x})(x,\,t)|^2{\rm d}t\Big)^{\frac{1}{2}}\\
&=:&{\rm D}_1f(x)+{\rm D}_2f(x)+{\rm D}_3f(x).
\end{eqnarray*}
Lemma \ref{l3.2} and the estimate (2.11) tells us that for any $p\in (1,\,\infty)$,
$$\|{\rm D}_1f\|_{L^p(\mathbb{R}^n)}+\|{\rm D}_2f\|_{L^p(\mathbb{R}^n)}\lesssim l\|f\|_{L^p(\mathbb{R}^n)}.$$
On the other hand, we have
\begin{eqnarray*}
{\rm D}_3f(x)&\le &\sum_{j\in\mathbb{Z}:\,{\rm diam}\,Q/4\leq 2^j\leq 4{\rm diam}\,Q}\sup_{t\in[1,2]}\int_{\mathbb{R}^n}|K_t^j*\phi_{l-j}(x-y)||f(y)|{\rm d}y\\
&\lesssim&M_{q'}f(x).
\end{eqnarray*}
By estimates for ${\rm D}_1$, ${\rm D}_2$ and ${\rm D}_3$, we obtain that
$$
\|\widetilde{\mathcal{M}}_{\Omega}^l(f\chi_{\mathbb{R}^n\backslash B_x})\|_{L^{q',\,\infty}(\mathbb{R}^n)}\lesssim l\|f\|_{L^{q'}(\mathbb{R}^n)}.
$$
This, along with the estimates (2.12) and (2.14), shows that the operator $S^l$ is bounded from $L^{q'}(\mathbb{R}^n)$ to $L^{q',\,\infty}(\mathbb{R}^n)$ with bound $Cl$.

We  now conclude the proof of Lemma \ref{l3.4}. By Lemma \ref{l3.3}, we see that for each $f\in L^{q'}(\mathbb{R}^n)$ with compact support, there exists  sparse families $\mathcal{S}_1,\dots,\,\mathcal{S}_{3^n}$, such that
\begin{eqnarray}\widetilde{\mathcal{M}}_{\Omega}^lf(x)\lesssim l\sum_{j=1}^{3^n}\mathcal{A}_{\mathcal{S}_j,\,q'}f(x).\end{eqnarray}
For each fixed $j=1,\,\dots,\,3^n$, it follows from Lemma \ref{l2.1} that
\begin{eqnarray*}
\|\mathcal{A}_{\mathcal{S}_j,\,q'}f\|_{L^p(\mathbb{R}^n,\,w)}=\big\|\mathcal{A}_{\mathcal{S}_j}^{\frac{1}{q'}}(|f|^{q'})\big\|_{L^{\frac{p}{q'}}(\mathbb{R}^n,\,w)}^{\frac{1}{q'}}
\lesssim&[w]_{A_{p/q'},\,p}\|f\|_{L^p(\mathbb{R}^n,\,w)}.
\end{eqnarray*}
This, via the estimate (2.15) leads to our desired conclusion.
\end{proof}

{\it Proof of Theorem 1.2}. Without loss of generality, we may assume that $\|\Omega\|_{L^{q}(S^{n-1})}=1$. By (2.7), we know that
\begin{eqnarray}\|\widetilde{\mathcal{M}}_{\Omega}^{2^{l}}f-\widetilde{\mathcal{M}}_{\Omega}^{2^{l+1}}f\big\|_{L^2(\mathbb{R}^n)}
\lesssim 2^{-\theta 2^{l}}\|f\|_{L^2(\mathbb{R}^n)},\end{eqnarray} and the series
$$\widetilde{\mathcal{M}}_{\Omega}=\sum_{l=1}^{\infty}(\widetilde{\mathcal{M}}_{\Omega}^{2^{l+1}}-\widetilde{\mathcal{M}}_{\Omega}^{2^{l}})
+\widetilde{\mathcal{M}}_{\Omega}^{2}$$
converges in the $L^2(\mathbb{R}^n)$ operator norm. Let $p\in (q',\,\infty)$ and $w\in A_{p/q'}(\mathbb{R}^n)$, by
\cite[Corollary 3.16 and Corollary 3.17]{hyta}, we know that for $\epsilon=c_n/(w)_{A_{p/q'}}$ with $c_{n}$ a constant depending only on $n$,  $w^{1+\epsilon}\in A_{p/q'}(\mathbb{R}^n)$,
$$[w^{1+\epsilon}]_{A_{p/q'}}\lesssim 4[w]_{A_{p/q'}}^{1+\epsilon},$$
and
$$[w^{1+\epsilon}]_{A_{\infty}}\lesssim [w]_{A_{\infty}}^{1+\epsilon},\,\,[w^{(1-(\frac{p}{q'})')(1+\epsilon)}]_{A_{\infty}}\lesssim [w^{1-(\frac{p}{q'})'}]_{A_{\infty}}^{1+\epsilon}.$$
Therefore,
$$\{w^{1+\epsilon}\}_{A_{p/q'},\,p}\lesssim \{w\}_{A_{p/q'},\,p}^{1+\epsilon}.$$
Lemma \ref{l3.4} tells us that
\begin{eqnarray}\big\|\widetilde{\mathcal{M}}_{\Omega}^{2^{l}}f-\widetilde{\mathcal{M}}_{\Omega}^{2^{l+1}}f\big\|_{L^p(\mathbb{R}^n,\,w^{1+\epsilon})}
\lesssim 2^{l}\{w^{1+\epsilon}\}_{A_{p/q'},\,p}\|f\|_{L^p(\mathbb{R}^n,\,w^{1+\epsilon})}.
\end{eqnarray}
On the other hand, by interpolating the estimates (2.16) and (2.17) with $w=1$, we know that for some $\varrho=\varrho_{p}\in (0,\,1)$,
\begin{eqnarray}\big\|\widetilde{\mathcal{M}}_{\Omega}^{2^{l}}f-\widetilde{\mathcal{M}}_{\Omega}^{2^{l+1}}f\big\|_{L^p(\mathbb{R}^n)}\lesssim 2^{-\varrho2^{l}}\|f\|_{L^p(\mathbb{R}^n)}.\end{eqnarray}
By interpolation with changes of measures (see \cite{stw}), we deduce from (2.17) and (2.18) that
\begin{eqnarray*}\big\|\widetilde{\mathcal{M}}_{\Omega}^{2^{l}}f-\widetilde{\mathcal{M}}_{\Omega}^{2^{l+1}}f\big\|_{L^p(\mathbb{R}^n,\,w)}
\lesssim 2^{l}2^{-\varrho\frac{\epsilon}{1+\epsilon}2^l}\{w\}_{A_{p/q'},p}\|f\|_{L^p(\mathbb{R}^n,\,w)}.
\end{eqnarray*}
As in  \cite{hyta}, a trivial computation involving the inequality
${\rm e}^x\geq x^2/2$,
now shows that
$$\sum_{l=1}^{\infty}2^{l}2^{-\varrho 2^{l}\frac{\epsilon}{1+\epsilon}}\lesssim\sum_{l: 2^{l}\leq \epsilon^{-1}}2^{l}+\sum_{l: 2^{l}> \epsilon^{-1}}2^{l}\big(\frac{2^{l}\epsilon}{1+\epsilon}\big)^{-2}\lesssim (w)_{A_{p/q'}}.
$$
We finally get that
\begin{eqnarray*}\|\widetilde{\mathcal{M}}_{\Omega}f\|_{L^p(\mathbb{R}^n,\,w)}&\le & \|\widetilde{\mathcal{M}}_{\Omega}^2f\|_{L^p(\mathbb{R}^n,\,w)}+\sum_{
l=1}^{\infty}\big\|\widetilde{\mathcal{M}}_{\Omega}^{2^{l+1}}f-\widetilde{\mathcal{M}}_{\Omega}^{2^{l}}f\big\|_{L^p(\mathbb{R}^n,\,w)}\\
&\lesssim &\{w\}_{A_{p/q'},\,p}(w)_{A_{p/q'}}\|f\|_{L^p(\mathbb{R}^n,\,w)}.
\end{eqnarray*}
This completes the proof of Theorem \ref{t1.1}.
\qed

\end{document}